\documentclass[11pt]{article}

\usepackage{amssymb}
\usepackage{amsmath}
\usepackage{amsfonts}
\usepackage{amsthm}
\newtheorem{theorem}{Theorem}
\newtheorem{example}[theorem]{Example}
\newtheorem{algorithm}[theorem]{Algorithm}

\newtheorem{definition}[theorem]{Definition}
\newtheorem{lemma}[theorem]{Lemma}

\begin{document}
\begin{center}{\Large Desingularization of function fields}\end{center}
\begin{center}{\large Douglas A. Leonard\\
                      Department of Mathematics and Statistics\\
                      Auburn University}\end{center}

\begin{abstract}
  This is a self-contained purely algebraic treatment of desingularization
  of fields of fractions $\mathbf{L}:=Q(\mathbf{A})$ of $d$-dimensional domains of the form
  \[\mathbf{A}:=\overline{\mathbf{F}}[\underline{x}]/\langle b(\underline{x})\rangle\]
  with a purely algebraic objective of uniquely describing $d$-dimensional valuations
  in terms of $d$ explicit (independent) local parameters and $1$ (dependent) local unit,
  for arbitrary dimension $d$ and arbitrary characteristic $p$.
  
  The desingularization will be given as a rooted tree with nodes labelled by domains $\mathbf{A}_k$
  (all with field of fractions $Q(\mathbf{A}_k)=\mathbf{L}$),
  sets $EQ_k$ and $INEQ_k$ of equality constraints and inequality constraints,
  and birational change-of-variables maps on $\mathbf{L}$.

  The approach is based on d-dimensional discrete valuations and local monomial orderings
  to emphasize formal Laurent series expansions in $d$ independent variables.
  It is non-standard in its notation and perspective.
\end{abstract}

\newpage
\section{Introduction}
Since this is a non-standard purely algebraic perspective of desingularization,
we limit our references to standard methods to
Cutkosky \cite{Cut}, Hartshorne \cite{Hart}, Harris \cite{Harr},
Kollar \cite{Koll}, and the WIKI page for {\em resolution of singularities},
as gateways to the standard literature available.

We shall start with the basic setup of a d-dimensional algebraic function field,
give a generalization of discrete valuations to d-dimensional discrete valuations,
use them to define coordinate values,
and then describe the algebraic objectives of desingularization
as we see them from this perspective, namely
to uniquely coordinatize valuations 
and describe them in terms of $d$ explicit (independent) local parameters
and a (dependent) local unit.

To produce such output, we'll use local (so series-based) monomial orderings,
initial monomials relative to such,
minimal weight sequences,
unimodular matrices, and unimodular transformations,
to get birational change-of-variables maps between various domains of the common function field.

Then we'll put this all in the form of a rooted desingularization tree.
Each node represents a domain of the function field together with equality and inequality constraints defining a part
of a partition. Arcs from it are birational change-of-variables maps that either refine the partition or the local
description of valuations on that part.
Sets of valuations are repartitioned until the valuations in each part have been uniquely coordinatized and have
$d$ explicit independent local parameters and a local unit described
by an irreducible polynomial in {\em strongly resolved form} as defined below.
This is summarized in the arc algorithm section at the end, outlining
how to move from the node label $(b_k,EQ_k,INEQ_k)$ using birational maps  $(\phi_{k,l},\psi_{l,k})$
describing a birational change-of-variables to a new node label $(b_l,EQ_l,INEQ_l)$.
[The idea of using both equality constraints and inequality constraints to define disjoint parts
(rather than just equality constraints to define varieties) was introduced in \cite{Leon}
in the context of elimination and extension to describe varieties.
The use of coordinates coming from $\mathbf{P}^1(\overline{\mathbf{F}})$ was also introduced there.]

Along the way we'll talk about global parameters, reduction, strongly resolved forms, and other useful algebraic concepts.
But we'll try not to use any terminology that would suggest geometrical, topological, or analytical content,
so as to avoid nonproductive or counterproductive preconceptions based on such ideas.

\newpage
\section{Motivational example}
It is extremely important to have a small but non-trivial example
to motivate this perspective, the new notation, and the algebraic objectives.
After all, if we cannot fully understand what can be expected in such examples,
there is no need to proceed further.

Our basic motivational example starts with the irreducible polynomial
\[b:=x_0^3+x_0x_1+x_1^5\]
defining the domain $A:=\overline{\mathbf{F}}[x_0,x_1]/\langle b\rangle$
and its field of fractions $\mathbf{L}:=Q(A)$.

There are {\em Weil divisors}
\[((x_0))=-5\cdot P_1+1\cdot P_2+4\cdot P_3, \quad
  ((x_1))=-3\cdot P_1+2\cdot P_2+1\cdot P_3\]
describing the zeros and poles of the functions $x_0,x_1\in\mathbf{L}$.
These suggest what formal Laurent series expansions for $x_0$ and $x_1$ should look like,
namely that the coefficient corresponds to the leading (so smallest) exponent in the series.

Given these divisors, a best result is of the form
\[ x_0=t_1^{-5}u_1^{-2},\ x_1=t_1^{-3}u_1^{-1},\ b_1:=1+u_1+t_1^7u_1^3\]
\[ x_0=t_2,\ x_1=t_2^2u_2,\ b_2:=1+u_2+t_2^7u_2^5\]
\[ x_0=t_3^4u_3,\ x_1=t_3,\ b_3:=1+u_3+t_3^7u_3^3\]
with $t_i$ a local parameter, $u_i$ a local unit and $b_i(t_i,u_i)=0$
the induced irreducible polynomial relation at $P_i$, whatever $P_i$ is.
Elsewhere, so for $(a_0,a_1)\in\overline{\mathbf{F}}^2\backslash\{(0,0)\}$
and $a_0^3+a_0a_1+a_1^5=0$,
\[ x_i=a_i+t,\ x_j=a_j+t^ku\]
for $\{i,j\}=\{0,1\}$, $k\in\{1,2\}$, $t$ a local parameter, $u$ a local unit,
and some induced irreducible polynomial relation $\overline{b}(t,u)=0$.

So the questions to be asked are how do we produce this result for this example,
then for examples in dimension $d=1$ in general,
and finally for examples in arbitrary dimension $d$ (and arbitrary characteristic $p$)?

\newpage
\section {The algebraic function field $\mathbf{L}$}

Start with $\mathbf{F}$ denoting either the {\em rational field} $\mathbf{Q}$ in characteristic $0$
or the {\em finite field} $\ \mathbf{F}_p$ in positive prime characteristic $p$,
with $\overline{\mathbf{F}}$ denoting the {\em algebraic closure}.
Let $b(x_0,\ldots,x_d)\in\overline{\mathbf{F}}[x_0,\ldots,x_d]$
be an {\em irreducible polynomial}, and assume the ideal $\mathbf{I}(x_0,\ldots,x_d)$
that it generates is the ideal of all relations among the variables.
The corresponding {\em quotient ring}
\[ \mathbf{A}:=\overline{\mathbf{F}}[x_0,\ldots,x_d]/\langle b(x_0,\ldots,x_d)\rangle\]
is then a {\em domain}.
Let its {\em field of fractions} $Q(\mathbf{A})$ be denoted by $\mathbf{L}$
and be called an {\em algebraic function field}.

Before proceeding further, 
it is necessary to make a basic assumption about $\mathbf{L}$
that is to be our universal object,
namely that it is a field in the {\em algebraic} sense that 
if $f\in\mathbf{L}\backslash\{ 0\}$, then
not only is $f^{-1}\in\mathbf{L}\backslash\{ 0\}$, 
but $f\cdot f^{-1}\equiv 1\in\mathbf{L}$.
Why should this be an obvious requirement? 
There is little use forcing $\mathbf{A}$ to be a domain 
so as to have no zero-divisors and hence cancellation, 
if its field of fractions does not have cancellation as well.

Both $\mathbf{A}$ and $\mathbf{L}$ will be said to have {\em dimension} $d$,
in that the ideal of all relations $I(x_1,\ldots,x_d)=\{ 0\}$ 
but the ideal of relations $I(x_0,\ldots,x_d)\neq\{ 0\}$.
Note here that for any finite subset of elements of $\mathbf{L}$
there is an ideal of all the induced relations among those elements,
but that this is {\em independent of any computation that produces it}.

We wish to talk about values of functions, but rather than evaluating
$f\in\mathbf{F}$ by writing it as a quotient $g/h$ of two polynomials
$g,h\in\overline{\mathbf{F}}[x_0,\ldots,x_d]$ and evaluating them independently,
we'll be relying on the valuations defined below to produce a value.
So as not to have to write $\underline{\nu}(f)$ for the valuation of $f$ at $\underline{\nu}$
and $f(\underline{\nu})$ for the value of  $f$ at $\underline{\nu}$,
we'll think of values as coordinates and write $coord(f,\underline{\nu})$.

[Note, not surprisingly, that function fields of dimension $d=1$ are markedly different from those of higher dimension
in that all functions $f\in\mathbf{L}$ can be evaluated at all valuations $\nu$.
Even in the simplest example for $d=2$, namely $\mathbf{L}:=\overline{\mathbf{F}}(x_0,x_1)$,
there is a valuation at which $coord(x_0)=0=coord(x_1)$ at which $x_2:=x_1/x_0$ can't be evaluated;
and a valuation at which $coord(x_2)=0=coord(x_1)$ at which $x_0$ can't be evaluated.]

\newpage
\section{Coordinate systems}
In our motivational example above, were we to limit ourselves to affine coordinates
(usually written $A^n$ but here written $\overline{\mathbf{F}}^n$),
we could not coordinatize $P_1$.

The standard fix for this is to use projective coordinates
(elements of $\mathbf{P}^n(\overline{\mathbf{F}})$).
Of course these are not coordinates in the same sense,
but rather, at best, ratios of such.
They can be used to define an affine cover so that affine coordinates
can be used in each of the $n+1$ affine charts.
They are gotten by replacing $x_i$ by $X_i/H$ and homogenizing equations.
The $n+1$ charts are then gotten by setting one of the $X_i$ or $H$ equal to $1$. 

We'll argue instead that the natural fix is to homogenize variables
by replacing $x_i$ symbolically by $g_i/h_i$ to get multi-homogeneous equations.
It should be the case that $coord(g/h)=\alpha/\beta$ if $coord(h/g)=\beta/\alpha$, even for $\beta=0$ and $\alpha\neq 0$,
so $coord(x_i)\in\mathbf{P}^1(\overline{\mathbf{F}})$ is the natural generalization of $\overline{\mathbf{F}}$
and hence  $(\mathbf{P}^1(\overline{\mathbf{F}}))^n$ is the natural generalization of $\overline{\mathbf{F}}^n$ in this context.

This allows us to view functions as truly independent.
Consider $\mathbf{L}:=\overline{\mathbf{F}}(x_0,x_1)$ used above.
If $x_0,x_1$ are truly independent, and allowed to take on the value $1/0$,
then rewriting this as $\mathbf{L}:=\overline{\mathbf{F}}(X_0/H,X_1/H)$
makes it hard to have $coord(x_0)=1/0$ and $coord(x_1)=1/1$.
This is not the case for $\mathbf{L}:=\overline{\mathbf{F}}(g_0/h_0,g_1/h_1)$
in that it is trivial to let $coord(g_0/h_0)=1/0$ and $coord(g_1/h_1)=1/1$.

[Later we may comment on other drawbacks of projective space.
Suffice it to say that  $\mathbf{P}^n(\overline{\mathbf{F}})$
will play no role in our theory for any $n>1$.]

\newpage
Start the desingularization of the motivational example with the multi-homogeneous polynomial
\[b^*(g_0,h_0,g_1,h_1):=g_0^3h_1^5+g_0h_0^2g_1h_1^4+h_0^3g_1^5\]
shorthand for $2^{d+1}=4$ irreducibles:
\begin{enumerate}
   \item \[b_0\left(\frac{g_0}{h_0},\frac{g_1}{h_1}\right)=\left(\frac{g_0}{h_0}\right)\left(\frac{g_1}{h_1}\right)+\left(\frac{g_0}{h_0}\right)^3+\left(\frac{g_1}{h_1}\right)^5\]
   \item \[b_1\left(\frac{g_0}{h_0},\frac{h_1}{g_1}\right)=1+\left(\frac{g_0}{h_0}\right)\left(\frac{h_1}{g_1}\right)^4+\left(\frac{g_0}{h_0}\right)^3\left(\frac{h_1}{g_1}\right)^5\]
   \item \[b_2\left(\frac{h_0}{g_0},\frac{g_1}{h_1}\right)=1+\left(\frac{h_0}{g_0}\right)^2\left(\frac{g_1}{h_1}\right)+\left(\frac{h_0}{g_0}\right)^3\left(\frac{g_1}{h_1}\right)^5\]
   \item \[b_3\left(\frac{h_0}{g_0},\frac{h_1}{g_1}\right)=\left(\frac{h_0}{g_0}\right)^3+\left(\frac{h_1}{g_1}\right)^5+\left(\frac{h_0}{g_0}\right)^2\left(\frac{h_1}{g_1}\right)^4\]
\end{enumerate}    

Instead of using an affine cover of overlapping sets, partition the multi-homogenous variety $V^*(b^*)$ into (disjoint) parts
\begin{enumerate}
   \item \[ S_0=\left\{ ((a_0:1),(a_1:1))\in(\mathbf{P}^1(\overline{\mathbf{F}}))^2\ :\ a_0^3+a_0a_1+a_1^5=0 \right\}\]
   \item \[S_1=\left\{ ((a_0:1),(1:0))\in(\mathbf{P}^1(\overline{\mathbf{F}}))^2\ :\ a_0^30^5+a_00^4+1=0\right\}=\emptyset\]
   \item \[S_2=\left\{ ((1:0),(a_1:1))\in(\mathbf{P}^1(\overline{\mathbf{F}}))^2\ :\ 1+0^2a_1^4+0^3a_1^5=0\right\}=\emptyset\]
   \item \[S_3=\left\{ ((1:0),(1:0))\in(\mathbf{P}^1(\overline{\mathbf{F}}))^2\ :\ 0^5+0^20^4+0^3=0\right\}=\{ (1/0,1/0)\}\]
\end{enumerate}
and deal with each of the non-empty affine problems produced separately.

In general start the desingularization tree with root node labelled by the multi-homogeneous polynomial
\[ b^*(g_0,h_0,\ldots,g_d,h_d)\]
and produce arcs to the related affine polynomial subproblems
$0\leq K< 2^{d+1}$, by writing $K$ in binary as
\[ K=\sum_{j=0}^dK_j2^j,\ K\in\{ 0,1\}\]
labelling the arcs with birational change-of-variables maps
$(\phi_K,\psi_K)$ defined by
\[ \psi_K(x_{K,j}):=(g_j/h_j)^{(-1)^{K_j}},\ \phi_K(g_j/h_j):=x_{K,j}^{(-1)^{K_j}}.\]
The corresponding irreducibles are then $b_K(x_{K,0},\ldots,x_{K,d})$,
with affine sets
\[EQ(K):=\{ (a_{K,0},\ldots a_{K,d})\in\overline{\mathbf{F}}^{d+1}\ :\ a_{K,j}=0\mbox{ if }K_j=1\mbox{ and }b_K(a_{K,0},\ldots a_{K,d})=0\}\]
and $INEQ(K):=\emptyset$ defining a part.

\newpage    
\section{d-dimensional valuations}
Valuations are usually defined as 1-dimensional maps;
at $\nu$, $\nu(f)>0$ is interpreted as $f$ having a zero of this order,
$\nu(f)=0$, interpreted as $f$ being a unit, and $\nu(f)<0$, intepreted
as $f$ having a pole of order $-\nu(f)$.
If this is thought of in terms of the leading exponent of a formal Laurent series expansion,
consider the following d-dimensional generalization instead
(even though a d-dimensional formal Laurent series may not have a leading exponent).
This is a crucial object in this algebraic approach.

\begin{definition}[d-dimensional valuations] 

The map $\underline{\nu}\ :\ \mathbf{L}\backslash\{ 0\}\to\mathbf{Z}^d$ 
is a {\em valuation} of a d-dimensional algebraic function field, iff it satisfies:
\begin{enumerate}
\item $\underline{\nu}(c)=\underline{0}$ for $c\in\overline{\mathbf{F}}\backslash\{ 0\}$;
\item $\underline{\nu}(f_1f_2)=\underline{\nu}(f_1)+\underline{\nu}(f_2)$;
\item if $\nu_i(f_1)< \nu_i(f_2)$, then
 $\nu_i(f_1-f_2)=\nu_i(f_1)$;
\item if  $f_2\neq f_1$ but $\nu_i(f_2)=\nu_i(f_1)$, then
 there exists $u\in\mathbf{L}$ such that $f_1=uf_2$ and $\nu_i(1-u)>0$;
\item there exist $d$ {\em local parameters} with $\nu_i(t_j)=\delta_{i,j}$
for $1\leq i,j\leq d$ $($so they are {\em independent}$)$.
\end{enumerate}
\end{definition}

[Warning: One could define a local ring or a discrete valuation ring as
\[\mathbf{O}_{\underline{\nu}}:=\{ 0\}\cup\{ f\in\mathbf{L}\ :\ \underline{\nu}(f)\succeq \underline{0}\}\]
with unique maximal ideal
\[m:=\langle t_1,\ldots, t_d\rangle;\]
but $f\in m$  does not imply that $\underline{\nu}(f)\succ \underline{0}\}$,
only that the individual terms of $f$ have that property.]
   
\newpage
\section{Coordinates from valuations}
This is also a crucial concept in this algebraic approach!
To motivate evaluation (that is a coordinate value) 
of a function $f\in\mathbf{L}$ at a valuation,
consider that when $d=1$ 
if $t,u_1,u_2\in\mathbf{A}$ with $t(P)=0$, $u_i(P)=c_i\neq 0$,
then for $f:=(t^{a_1}u_1)/(t^{a_2}u_2)$ it makes sense for
\[ f(P):=\begin{cases}0/1\text{ if }a_1>a_2\\
                      c_1/c_2\text{ if }a_1=a_2\\
                      1/0\text{ if }a_1<a_2\end{cases}\]
meaning that $f$ has a zero of order $a_1-a_2$ at $P$, 
is a unit at $P$, or has a pole of order $a_2-a_1$ at $P$ respectively, 
rather than just giving up on evaluating $f$ at $P$ when $a_1,a_2>0$.

The {\em coordinate} $coord(f,\underline{\nu})$ will be based on the valuation $\underline{\nu}$ 
rather than by writing $f=g/h$ and evaluating $g(P),h(P)$ as polynomials and using $f(P)=g(P)/h(P)$.  

\begin{definition}[Coordinates]
 
\[coord(\ ,\underline{\nu})\ :\ \mathbf{L}\to \mathbf{P}^1(\overline{\mathbf{F}})\]
is defined by: 
\[coord(f,\underline{\nu}):=\begin{cases}
    c/1\text{ if there is a unique  $c\in\overline{\mathbf{F}}$ such that }
    f-c\in\langle t_1,\ldots,t_d\rangle;\\
    1/0\text{ if } coord(f^{-1},\underline{\nu})=0/1;\\
     \text{ undefined otherwise.}\end{cases}\]
A {\em function} $f$ $($so element of $\mathbf{L})$
 is {\em regular} at a valuation $\underline{\nu}$
 iff $coord(f,\underline{\nu})$ is defined.
\end{definition}

[Again, this is supposed to agree with the standard method of thinking of the elements of $\mathbf{L}$ 
as quotients $g/h$ of elements $g,h\in\mathbf{A}$ 
and finding coordinate values for both $g$ and $h$ to get one for $g/h$.
But it works as well when $g(P)=1$ and $h(P)=0$ as well as sometimes when $g(P)=0=h(P)$.]

Also realize that this idea of a function $f$ regular at a valuation $\underline{\nu}$
has {\em nothing} to do with $f$ considered anywhere except at that valuation.
[This is more palatable if one thinks about working in positive characteristic,
where the temptation to think non-algebraically is not so strong.]

\newpage
The purely algebraic goal of {\em desingularization} here is then to find {\em coordinate functions} 
such that

\begin{itemize}
\item  
\[(coord(x_{k,0},\underline{\nu}),\ldots,coord(x_{k,d},\underline{\nu}))
=(coord(x_{k,0},\underline{\mu}),\ldots,coord(x_{k,d},\underline{\mu}))\]
iff $\underline{\mu}\equiv\underline{\nu}$; 
\item and that there are $d$ independent, {\em explicit local parameters} $t_{j,\underline{\nu}}$
with $\nu_i(t_{j,\underline{\nu}})=\delta_{i,j}$,
each of the form 
\[x_{k,j}-coord(x_{k,j},\underline{\nu}),\ 1\leq j\leq d.\]
\end{itemize}

\newpage
\section{Global parameters and reduction}

The best possible result of desingularization of $\mathbf{L}$ is
to have $\mathbf{L}=\overline{\mathbf{F}}(t_1,\ldots,t_d)$
and a {\em global birational parameterization} $(\Phi,\Psi)$
describing the given variables in terms of these new independent variables and vice versa.

[As a warning there may be several ways to do this,
as in the example $b:=wx-yz=0$,
wherein any variable can be solved for in terms of the others.]

The next best possible result is to find a domain
\[ A_k:=\overline{\mathbf{F}}[x_{k,0},\ldots, x_{k,d}]/
  \langle b_k(x_{k,0},\ldots,x_{k,m})\rangle\]
for some $m<d$ so that $x_{k,m+1},\ldots,x_{k,d}$
are identified as {\em global parameters}.
This at least reduces the desingularization problem to one in dimension $m<d$.
Surprisingly many examples in the literature have at least one global variable, if not $d$ such.
\begin{theorem} 

If $b(x_0,\ldots,x_d)=f_1(x_{0},\ldots,x_{d-1})-x_df_2(x_{0},\ldots,x_{d-1})$, 
then $x_d$ is merely a variable name for  
$f_1(x_{0},\ldots,x_{d-1})/f_2(x_{0},\ldots,x_{d-1})\in\mathbf{L}$.
So $\mathbf{L}=\overline{\mathbf{F}}(x_{0},\ldots,x_{d-1})$.
\end{theorem}

This would seem to be an observation more than a theorem,
with the proof essentially given by the statement.
Nevertheless there are examples in the literature of this form.
Kollar 3.3.4 is the {\em quadratic cone} $uv-w^2=0$, so $u=w^2/v$.
Kollar 3.6.1 is the {\em pinch point} or {\em Whitney umbrella} $x^2-y^2z=0$, so $z=x^2/y^2$.
[Admittedly these examples were given to exemplify some aspects of standard theory that are irrelevant to the theory here.]

\begin{theorem}
If
\[b(x_0,\ldots,x_d)=\sum_{j=0}^mx_d^{m-j}g^j(x_{0},\ldots,x_{d-1})f_j(x_0,\ldots,x_d),\]
then the birational change-of-variables defined by
\[\phi(x_d):=y_dg(x_{0},\ldots,x_{d-1}),\quad \psi(y_d):=x_d/g(x_{0},\ldots x_{d-1})\]
produces a {\em reduction}
\[red(b)(x_{0},\ldots,x_{d-1},y_d):=\sum_{j=0}^my_d^{m-j}f_j(x_0,\ldots,x_{d-1},y_dg(x_{0},\ldots,x_{d-1})).\]
\end{theorem}

Whether this is considered a theorem, definition, or observation,
there are examples in the literature that could be simplified with just this.
A generalization of Kollar 3.6.2 $b:=x_0^2+x_1^2+x_2^{2m_1+r_1}x_3^{2m_2+r_2}=0$ for $r_i\in\{0,1\}$
together with an obvious variant of theorem 2 gives
$\phi(x_0):=y_0x_2^{m_1}x_3^{m_0}$, $\phi(x_1):=y_1x_2^{m_1}x_3^{m_0}$, and
$\phi(b)=x_2^{2m_1}x_3^{2m_0}red(b)$, with $red(b):=y_0^2+y_1^2+x_2^{r_1}x_3^{r_0}$.
Then if either $r_1=1$ or $r_0=1$, $x_2$ or $x_3$ can be solved for, using theorem 1.
If both are $0$, then this can be parameterized, as in the Eisenbud example below.
So this is globally parameterizable in all cases.

\begin{theorem}

If there are weights $w_i\in\mathbf{N}$ for each $x_i$ such that $weight(\underline{x}^{\underline{\alpha}})=w$
is the same for all $\underline{x}^{\underline{\alpha}}$ in the support of $b$,
then there is at least one global parameter.
\end{theorem}

{\bf Proof} There is a birational change-of-variables map induced by the weights (see the unimodular section below)
such that $\phi(b)=z_d^wb_1(z_{0},\ldots,z_{d-1})$.

Again this would seem to be a self-evident theorem, but there are examples in the literature on which this could be used.

Cutkosky's exercise 7.35.1 is an example cited from Narasimhan with $b:=x_0^2+x_1x_2^3+x_2x_3^3+x_1^7x_3=0$.
Even the weights $(32,7,19,15)$ are given.
The change of variables $\phi(x_0):=z_3^{32}z_2^4z_1$, $\phi(x_1):=z_3^{7}z_2$, $\phi(x_2):=z_3^{19}z_2^2z_0$, and $\phi(x_3):=z_3^{15}z_2^2$
produces $\phi(b)=z_3^{64}b_1$ for $b_1:=z_2^2+z_2z_1^2+z_2z_0+z_0^3$ in any characteristic, not just $p=2$ as in the example.

Cutkosky's exercise 7.35.2 is an example cited from Hauser with $b:=x_0^2+x_1^4x_2+x_1^2x_2^4+x_2^7=0$.
There are implicit weights $(7,3,2)$ not given there.
The change of variables $\phi(x_0):=z_2^{7}z_1^3$, $\phi(x_1):=z_2^3z_1z_0$, and $\phi(x_2):=z_2^2z_1$, produces $\phi(b)=z_2^{14}z_1^5b_1$
for $b_1:=z_1+z_0^4+z_1z_0^2+z_1^2$ in any characteristic, not just $p=2$ as in the example.

The simple example in Eisenbud $b:=x_0^2+x_1^2+x_2^2=0$ has an obvious weight function $(1,1,1)$.
Letting $\phi(x_0):=z_2$, $\phi(x_1):=z_2z_1$, and $\phi(x_2):=z_2z_0$ produces $\phi(b)=z_2^2b_1$ for $b_1:=1+z_0^2+z_1^2$.
Methods below can be used to rewrite this as $B_1:=(1+a_1^2+a_0^2)+2(a_0y_0+a_1y_1)+y_0^2+y_1^2$ with constant term $1+a_1^2+a_0^2=0$.
Then $\phi(y_0):=z_1$, $\phi(y_1):=z_0z_1$ gives $\phi(b_1)=z_1b_2$ for $b_2:=2(a_0+a_1z_0)+z_1(1+z_0^2)$.
Then $z_1:=-2(a_0+a_1z_0)/(1+z_0^2)$ produces a global parameterization in terms of $z_2$ and $z_0$ in any characteristic except $p=2$
(since $b$ is reducible in that case).

The following seems to be more useful in positive characteristic.
\begin{theorem}
If
\[ b(x_0,\ldots,x_d)=f_1(x_0,\ldots, x_d)^k-x_dg(x_0,\ldots,x_{d-1})f_2(x_0,\ldots,x_d)^k\]
then $\phi(x_d):=z_d^k/g(x_0,\ldots,x_{d-1})$ and $\phi(x_i):=z_i$ otherwise, produces
\[ \phi(b)=f_1(z_0,\ldots,z_{d-1},z_d^kg(x_0,\ldots,x_{d-1}))-z_df_2(z_0,\ldots,z_{d-1},z_d^kg(x_0,\ldots,x_{d-1})).\]
\end{theorem}

Again this could be considered only an observation. But it applies to the Hauser example above,
rewritten as $b:=(x_2+x_1x_0^2)^2+x_0(x_1^2+x_0^3)^2=0$ in characteristic $p=2$.
That produces $b_1:=(z_2+z_1z_0^2)+z_0(z_1^2+z_0^3)$ from which $z_2=z_1z_0^2+z_0z_1^2+z_0^4$ gives a global parameterization.

A more serious example is given by
\[b_0:=(x_0x_1x_2x_3)^5+x_0^{12}x_1^8x_2^4+x_1^{12}x_2^8x_3^4+x_2^{12}x_3^8x_0^4+x_3^{12}x_0^8x_1^4\]
which is messy. But in characteristic $2$ this reduces to
\[b_0:=(x_0x_1x_2x_3)^5+(x_0^3x_1^2x_2+x_1^3x_2^2x_3+x_2^3x_3^2x_0+x_3^3x_0^2x_1)^4.\]

Then the theorem above gives three birational change-of variables maps defined by:
$x_0:=t_0^4/(x_1x_2x_3)$, $x_1:=t_1^2/(x_3t_0)$, $x_2:=t_2^2/(x_3t_1)$,
followed by $x_3:=t_0s_3$, $t_2:=t_0^2s_2$, $t_1:=t_0^2s_1$, that finally
produce $t_0=s_3(s_3s_1+s_2^2+s_2s_1+s_1)/(s_2^2s_1^2)$, and hence a global
parameterization.
This is a simple hand computation, as opposed to a much more difficult computer computation absent this theorem.

\newpage
\section{Series-based monomial orderings}

Standard {\em global monomial orderings} highlight, among other things,
the larger monomials, whatever larger means;
while standard {\em local monomial orderings} highlight the smaller monomials.
Therefore, better terms might be {\em polynomial-based orderings} and {\em series-based monomial orderings}.
What is important here is formal series expansions that are consistent with polynomial relations,
so it makes more sense to pick a monomial ordering that highlights series than one that highlights polynomials.
This is the crucial step that explains the difference between a weak resolution and a strong one.

In general, choose a generic element $\underline{a}$
satisfying the equalities defined by $EQ(k)$ and the inequalities defined by $INEQ(k)$,
define $y_j:=x_{k,j}-a_{k,j}$ and rewrite $b_k(x_{k,0},\ldots,x_{k,d})$ as
$B_k(\underline{a},\underline{y})\in(\overline{\mathbf{F}}[\underline{a}])[\underline{y}]$
with a local monomial ordering on $\underline{y}$.
Note that the choice of $\underline{a}$ forces the constant term to vanish.
Not surprisingly, the coefficients are related to mixed partials of $b_k$ evaluated at $\underline{a}$,
as they would be in Taylor series in several variables. 

In the ongoing motivational example this produces
\[ B_0:=(a_{0,1}+3a_{0,0}^2)y_0+(a_{0,0}+5a_{0,1}^4)y_1\]
\[+(3_{0,0})y_0^2+y_0y_1+(10a_{0,1}^3)y_1^2\]
\[+y_0^3+(10a_{0,1}^2)y_1^3+(5a_{0,1})y_1^4+y_1^5\]
and
\[ B_3:=y_0^3+y_1^5+y_0^2y_1^4.\]

[The standard {\em Jacobian criterion} in this context would be that
$\underline{a}$ is not {\em smooth} iff all the linear coefficients vanish.
But {\em smoothness} is not the goal here, as it is only a {\em weak} version of desingularization.
The example on the WIKI page cited earlier explains that a {\em strong desingularization} doesn't end with smoothness,
but continues until there are {\em simple normal crossings}.
This is probably the objective here as well, but is a more geometric definition.
So consider an alternative purely algebraic definition of a stopping
criterion called {\em strongly resolved form} below.
And note that the Jacobian criterion will not be used in this theory except possibly for comparison sake.]

\newpage
\section{Initial monomials and partitions}

A monomial $\underline{y}^{\underline{\alpha}}$ occuring in $B$
is an {\em initial monomial} iff
no other such monomial $\underline{y}^{\underline{\beta}}$
of $B$ divides $\underline{y}^{\underline{\alpha}}$.

So partition the generic $\underline{a}$ into (disjoint) parts
according to what the set of initial monomials of $B$ would be.

In the motivational example, there are $2$ parts for $B_0$
depending on whether $a_{0,1}=0=a_{0,0}$ or $a_{0,0}\neq 0\neq a_{0,1}$,
but only one part for $B_3$.

\section{Minimal weight sequences}

The sequence $\underline{w}:=(\nu_d(y_0),\ldots \nu_d(y_d))$
is a {\em weight sequence} for $B$ iff there are distinct monomials
$\underline{y}^{\underline{\alpha}}$ and $\underline{y}^{\underline{\beta}}$
of $B$ such that
\[\nu_d(\underline{y}^{\underline{\alpha}})=\nu_d(\underline{y}^{\underline{\beta}})\leq \nu_d(\underline{y}^{\underline{\gamma}})\]
for all monomials $\underline{y}^{\underline{\gamma}}$ of $B$.

The weight sequence $\underline{w}$ is {\em minimal} iff it is not the sum of two smaller weight sequences for $B$.

In the motivational example there are minimal weight sequences $(1,1)$, $(1,2)$ and $(2,1)$ for one part,
either $(1,2)$ or $(4,1)$ respectively for the other part of $B_0$,
and $(5,3)$ for the part of $B_3$.

The reason to look at only minimal weight sequences is that otherwise we could produce infinitely many examples by
replacing $t_i$ by $t_it_d^k$ for any $0<i<d$, all equivalent to the original. [Standard methods used by {\tt resolve}
generate several hundred charts for Kollar's example 2.65, and it takes time to reduce these to fewer such with different
minimal weight sequences to compare answers.]

\newpage
\section{Unimodular transformations}

While it is possible to desingularize using only algebraic blowups, 
this is analogous to relying on subtractions when there are divisions to be used to replace sequences of subtractions.
[Blowups are generally not directed in that some produce progress while others do not.
The real problem is that they are commonly used to blow up geometric objects.
That is, there may be a line $l$ of singularities known,
with a point $p$ on that line known to have a more complicated singularity than elsewhere on the line.
The Whitney umbrella mentioned above is such an example.
Rather than having to choose to blow up either $l$ or $p$,
the suggestion here is to partition so as to deal with $l^c$, $l\backslash p$, and $p$ separately.]

\begin{definition}[Unimodular transformations]
  A unimodular transformation is a birational change-of-variables map defined by unimodular matrices as follows.

Suppose that $\nu_d(x_i)>0$ for all $0\leq i< m$ and $D:=\gcd\{\nu_d(x_i)\ :\ 0\leq i< m\}$.
Then the extended euclidean algorithm
$($or equivalently row-reduction over the natural numbers$)$
can produce a unimodular matrix $M$
with first column  $(\nu_d(x_i)/D\ :\ 0\leq i< m)^T$.
Since it is unimodular, $M^{-1}$ has entries in $\mathbf{Z}$.
So there is a change of variables defined by
\[ \phi(x_i):=\prod_{j=0}^{m-1}y_j^{M_{i,j}},\ \quad  \psi(y_j):=\prod_{i=0}^{m-1}x_i^{(M^{-1})_{j,i}},\]
with the other variables left unchanged.

The term birational change-of-variables will be extended to mean a composition of
translating by a generic $\underline{a}$ followed by such a unimodular transformation.
\end{definition}

The general result is then:

\begin{lemma}[weights]
  If there are weights $wt(x_i)$ such that every monomial $\underline{x}^{\underline{\alpha}}$ 
occurring in $b$ has the same total weight $w:=\sum_i \alpha_iwt(x_i)$, 
then there is a unimodular transformation that produces an independent global parameter and reduces the desingularization dimension.
\end{lemma}

{\bf Proof} $\phi(b)=t_d^wf(t_0,\ldots,t_{d-1})$.

Surprisingly many examples in the literature are either {\em homogeneous} or {\em weighted-homogeneous},
so have at least one {\em independent global parameter}.
[Some of these examples in the literature even have the weights calculated in them as exemplified earlier,
but no effective use is then made of this information.]

\section{Birational change of variables maps on arcs}

For each part described by initial monomials and each minimal weight sequence corresponding to it,
we can now produce a translation by $\underline{a}$ followed by a unimodular transformation gotten from the weights
to produce {\em birational change-of-variables maps} $\phi_{k,l}\ :\ A_k\to A_l$ and $\psi_{l,k}\ :\ A_l\to A_k$
for the arc of the desingularization tree from node $k$ to node $l$.
In our continuing motivational example,

\[\phi_{0,4}(x_{0,0}):=a_{0,0}+x_{4,1},\ \phi_{0,4}(x_{0,1}):=a_{0,1}+x_{4,0}x_{4,1},\]
\[\phi_{0,5}(x_{0,0}):=a_{0,0}+x_{5,0}x_{5,1}^2,\ \phi_{0,5}(x_{0,1}):=a_{0,1}+x_{5,1},\]
\[\phi_{0,6}(x_{0,0}):=a_{0,0}+x_{6,1},\ \phi_{0,6}(x_{0,1}):=a_{0,1}+x_{6,0}x_{6,1}^4.\]

\[\phi_{0,4}(b_0)=x_{4,1}b_4\]
\[b_4:=(a_1+5a_0^4)+(a_0+3a_1^2)x_{4,0}+(10a_0^3)x_{4,1}+\]
\[x_{4,0}x_{4,1}+(3a_1)x_{4,0}^2x_{4,1}+(10a_0^2)x_{4,1}^2+x_{4,0}^3x_{4,1}^2+(5a_0)x_{4,1}^3+x_{4,1}^4\]
\[ \phi_5(b_0):=x_{5,1}^3b_5\]
\[b_5:=1+x_{5,0}+x_{5,0}^5x_{5,1}^7\]
\[\phi_{0,6}:=x_{6,1}^5b_6\]
\[ b_6:=1+x_{6,0}+x_{6,0}^3x_{6,1}^7\]

The next two cases $b_1$ (with $g_1/h_1=0/1$ and $h_0/g_0=0/1$) and $b_2$ (with  $h_1/g_1=0/1$ and $g_0/h_0=0/1$) are vacuous.

The case $b_3$ (with $h_1/g_1=0/1$ and $h_0/g_0=0/1$) has $EQ_3=ideal(a_{3,0},a_{3,1})$ 
\[ b_3:=x_{3,0}^3+x_{3,1}^5+x_{3,0}^2x_{3,1}^4=y_0^3+y_1^5+y_0^2y_1^4\]
This has  $init(b_3)=(y_0^2y_1^4,y_0^3,y_1^5)$ and minimal weight vector $(5,3)$.
\[\phi_{3,7}(x_{3,0}):=x_{7,0}^2x_{7,1}^5,\ \phi_{3,7}(x_{3,1}):=x_{7,0}x_{7,1}^3\]
\[\phi_{3,7}(b_3)=x_{7,0}^{5}x_{7,1}^{15}b_7\]
\[ b_7:=1+x_{7,0}+x_{7,0}^3x_{7,1}^7\]

\section{Strongly resolved form}

The {\em directed rooted tree of domains} produced has
\begin{enumerate}
    \item nodes labelled by the irredicible polynomials $b_l$
          defining the domain $\mathbf{A}_l$ (still with $Q(\mathbf{A}_l)=\mathbf{L}$);
          and polynomial equalities and polynomial inequalities
          defined by $EQ(l)$ and $INEQ(l)$, describing the generic coordinates corresponding to that node;
    \item arcs $(k,l)$ labelled by {\em birational change-of-variables maps} $(\phi_{k,l},\psi_{l,k})$.
\end{enumerate}

[It doesn't hurt to include {\em birational change-of-variables maps} $(\Phi_{l},\Psi_{l})$
between the root node and node $l$ at node $l$,
even though these are compositions of the maps on the path from the root node to node $l$.]

All we need is a stopping criterion.
Our goal was to produce $d$ independent explicit local parameters
$t_j:=x_{l,j}-coord(x_{l,j})$, $1\leq j\leq d$ and a local unit $x_{l,0}$.
For $x_{0,l}$ to be a unit, it should have a formal series expansion
in terms of the independent local parameters $t_1,\ldots,t_d$ with
non-zero constant term.
It would be nice to be able to produce that series somehow, even if only
recursively.

\begin{definition}
  The irreducible polynomial $b_0(x_0,\ldots,x_d)$ is in {\em strongly resolved form} iff
  mod $x_d$ it is linear in $x_0$.
  That is, it can be written as an element in the form
\[ f_0(x_1,\ldots,x_{d-1})+x_0f_1(x_1,\ldots,x_{d-1})+x_{d}D(x_1,\ldots,x_d)\sum_{j=0}^mx_0^jg_j(x_1,\ldots,x_d)\]
with $0=\nu_d(x_0)=\nu_d(f_1(x_1,\ldots,x_{d-1}))<\nu_d(x_d)$ and $\{\gcd(g_j)\ :\ 0\leq j\leq m\}=1$. 
\end{definition}

\begin{lemma}
  If $b_0(x_0,\ldots,x_d)$ is in this {\em strongly resolved form}, then not only is $x_0$ a local unit,
  but it is possible to recursively solve for its formal series expansion, an element of $\overline{\mathbf{F}}[[x_1,\ldots,x_d]]$.
\end{lemma}

\begin{proof}
    Use the birational change-of-variables defined by
    \[\psi(u_0):=(f_0(x_{1},\ldots,x_{d-1})+x_0f_1(x_1,\ldots,x_{d-1}))/(x_dD(x_{1},\ldots,x_d)),\]
    \[\phi(x_0):=(-f_0(x_1,\ldots,x_{d-1})+x_dD(x_1,\ldots,x_d)u_0)/f_1(x_1,\ldots,x_{d-1}).\]
  Then $f_1^m\phi(b_0)/(x_dD(x_1,\ldots,x_d))=u_0+\sum_{j=0}^m(-f_0+u_0x_dD)^jf_1^{m-j}g_j(x_1,\ldots,x_d)$.
  Then, $mod\ x_d$, this is $u_0+\sum_{j=0}^m(-f_0)^jf_1^{m-j}g_j(x_1,\ldots,x_{d-1},0)$, linear in $u_0$.
  \end{proof}

\newpage
\begin{example}[weak resolution versus strong resolution]

  Consider the irreducible polynomial $b_0:=x_{0,0}+x_{0,1}^2++x_{0,0}^2x_{0,1}+x_{0,2}^3+x_{0,0}^2x_{0,1}x_{0,2}\in\overline{\mathbf{F}}[x_{0,0},x_{0,1},x_{0,2}]$.
The monomial $x_{0,0}$ in the support of $b_0$ is what makes this satisfy the Jacobian criterion for smoothness.
But at $EQ_0=ideal(a_{0,0},a_{0,1},a_{0,2})$, $init(b_0)=(x_{0,0},x_{0,1}^2,x_{0,2}^3)$ has more than two elements,
so there should still be desingularization to do despite having $x_{0,0}$ of degree $1$.
Here there are three disjoint cases to consider:
\begin{enumerate}
   \item \[ 0<\nu_2(x_{0,2}^3)=\nu_2(x_{0,1}^2)\leq \nu_2(x_{0,0});\]
   \item \[ 0<\nu_2(x_{0,2}^3)=  \nu_2(x_{0,0})   < \nu_2(x_{0,1}^2);\]
   \item \[ 0<\nu_2(x_{0,1}^2)=  \nu_2(x_{0,0})   < \nu_2(x_{0,2}^3).\]
\end{enumerate}

This leads to three branches in the tree, leading to three leafs.

The first arc from node $0$ has unimodular transformation defined by:
\[ \phi_1(x_{0,2}):=x_{1,1}x_{1,2}^2,\ \phi_1(x_{0,1}):=x_{1,1}x_{1,2}^3,\ \phi_1(x_{0,0}):=x_{1,0}x_{1,1}^2x_{1,2}^6\]
with $\nu_2(x_{1,2})>0$, $\nu_2(x_{1,1})=0$, and $\nu_2(x_{1,0})=0$.
So\[ \phi_1(b_0)=x_{1,1}^2x_{1,2}^6b_1,\ b_1:=x_{1,0}+x_{1,1}+x_{1,0}x_{1,1}^3(1+x_{1,1})x_{1,2}^9\]
with $EQ_1=ideal(a_{0,0},a_{0,1},a_{0,2},a_{1,2},a_{1,0}+a_{1,1})$.

The second arc has unimodular transformation defined by:
\[ \phi_2(x_{0,2}):=x_{2,2},\ \phi_2(x_{0,0}):=x_{2,0}x_{2,2}^3,\ \phi_2(x_{0,1}):=x_{2,1}x_{2,2}^2\]
with $\nu_2(x_{2,2})>0$, $\nu_2(x_{2,0})=0$, $\nu_2(x_{2,1})=0$.
So\[ \phi_2(b_0)=x_{2,2}^3b_2,\ b_2:=1+x_{2,0} +x_{2,1}x_{2,2}+x_{2,0}x_{2,1}x_{2,2}^4(1+x_{2,2})\]
with $EQ_2=ideal(a_{0,0},a_{0,1},a_{0,2},a_{2,2},1+a_{2,0})$.

The third arc has unimodular transformation defined by:
\[ \phi_3(x_{0,1}):=x_{3,2},\ \phi_3(x_{0,0}):=x_{3,0}x_{3,2}^2,\ \phi_3(x_{0,3}):=x_{3,1}x_{3,2}\]
with $\nu_2(x_{3,2})=1$, $\nu_2(x_{3,0})=0$, and $\nu_2(x_{3,1})=0$.
So\[ \phi_3(b_0)=x_{3,2}^2b_3:=1+x_{3,0}+x_{3,1}^3x_{3,2}+x_{3,0}x_{3,2}^3(1+x_{3,1}x_{3,2})\]
with $EQ_3=ideal(a_{0,0},a_{0,1},a_{0,2},a_{3,2},1+a_{3,0})$.
\end{example}

\newpage
\section{Arc algorithm}
\begin{algorithm}[Strong desingularization]
Start with node $k$ labelled by $(b_k,EQ_k, INEQ_k)$
with $b_k$ not in strongly resolved form.

\begin{enumerate}
    \item Rewrite $b_k$ as $B_k\in (\overline{\mathbf{F}}[\underline{a}])[\underline{y}]$
          using a series monomial ordering on $\underline{y}$.
    \item Partition generic $\underline{a}$
          (so for which $f(\underline{a})=0$
          for all $f\in\langle EQ_k\rangle$
          but $f(\underline{a})\neq 0$
          for all $f\in\langle INEQ_k\rangle$)
          according to what $init(B_k)$ is.
    \item For each of those find all possible minimal weight sequences.
    \item For each such, produce a unimodular matrix.
    \item Label the arc $(k,l)$ with the corresponding
          birational change-of-variables pair $(\phi_{k,l},\psi_{k,l})$.
    \item Label node $l$ with $(b_l,EQ_l,INEQ_l)$.
\end{enumerate}
\end{algorithm}

\begin{theorem} If there is more than one arc from node $k$, then the set of valuations at node $k$
  will have been further refined.
  If there is only one arc, then the series expansions for the variables will have been improved.
\end{theorem}

A beta version of working {\sc Macaulay2} code,
called {\tt FunctionFieldDesingularization},
including examples,
is appended here after the end of document command
so that interested parties can cut-and-paste it
to test it and provide constructive feedback.
Updates will be posted on the author's website or through Macaulay2
once it is submitted there.

\end{document}